\documentclass[11pt,reqno]{amsart}

\usepackage{amssymb,amscd,amsxtra}
\usepackage{xypic}

\usepackage{hyperref}
\hypersetup{colorlinks=false}

\theoremstyle{plain}

\theoremstyle{plain}
\usepackage{amsfonts}
\usepackage{amssymb}
\usepackage{amsthm}
\usepackage{amsmath}
\usepackage{amsthm}
\usepackage{url}
\usepackage{amscd}
\usepackage{xypic}
\usepackage{array}
\newtheorem{theorem}{Theorem}
\newtheorem{corollary}{Corollary}
\newtheorem{lemma}{Lemma}
\newtheorem{proposition}{Proposition}
\newtheorem{definition}{Definition}

\newtheorem{remark}{Remark}

\theoremstyle{definition}

\theoremstyle{remark}

\newdimen\theight
\def\TeXref#1{%
             \leavevmode\vadjust{\setbox0=\hbox{{\tt
                     \quad\quad  {\small \textrm #1}}}%
             \theight=\ht0
             \advance\theight by \lineskip
             \kern -\theight \vbox to
             \theight{\rightline{\rlap{\box0}}%
             \vss}%
             }}%

\title{Mapping Spaces and Postnikov Invariants}

\author[M. Moreira]{Manuel F. Moreira G.}
\address{Departamento de Xeometr\'{\i}a e Topolox\'{\i}a\\
         Facultade de Matem\'aticas\\
         Universidade de Santiago de Compostela\\
         15782 Santiago de Compostela\\ Spain}
\email{morgal2002@gmail.com}
\thanks{The  author was partially supported by the grant ERASMUS MUNDUS LOTE~20B, for 
the period 2010--2012}

\date{\today}



\begin{document}

\maketitle

\begin{abstract}

If $q:Y\longrightarrow{B}$ is a fibration and $Z$ is a space, then
the free range mapping space $Y!Z$ has a collection of partial
maps from $Y$ to $Z$ as underline space, i.e. those such maps
whose domains are individual fibre of $q$.

It is shown in \cite{P. Booth6} that these spaces have applications to several
topics in homotopy theory. These such results are given in
complete detail, concerning identification, cofibrations and
sectioned fibrations.  The necessary topological foundations for 
two none complicated applications, i.e.  to the cohomology of fibrations and the
classification of Moore-Postinikov systems, are given,
and the applications themselves outlined.

The usual argument is in the context of the usually category of
all topological spaces, and this necessarily introduces some new
problems.  Whenever we  work with exponential laws for mapping
spaces, in that category, we will usually find that we are forced
to assume that some of the spaces are locally compact and
Hausdorff, which detracts considerable from the generality of the
results obtained.

In this paper we develop the corresponding theory in the category
of compactly generated or k-spaces, which is free of the
inconvenient assumptions referred to above. In particular, we
obtain the k-space version of the applications to
identifications, cofibrations and sectioned fibrations, and establish improved
foundations for the k-versions of the other two applications,
i.e. the cohomology of fibrations and a classification theory for
Moore-Postnikov factorizations.

\end{abstract}

\tableofcontents

\maketitle

\section{Compactly Genereted Spaces $k-spaces$}
\subsection{Universal Property for $k-spaces$}

Let $X$ be a topological space. We define $kX$ to be the space $X$
retopologized with the final topology  (see \cite{S. Willard})
relative to all incoming maps from compact Hausdorff spaces.
\newline
Thus if $$g:C\longrightarrow{X}$$ is continuous, then
$$g:C\longrightarrow{kX}$$ is continuous, in fact $kX$ has the
finest topology for which all maps $$g:C\longrightarrow{X}$$ are
continuous. If $kX=X$ then $X$ will be said to be a
\emph{compactly generated space} or $k-space$. We will refer to
$kX$ as the $k-ification$ of $X$. For more detail, concerning
$k-spaces$ see \cite{R. Brown2}.

\begin{theorem}{\textbf{Universal Property of the space $kX$}}
Let $X$, and $Y$ be spaces and $h:X\longrightarrow{Y}$ be a
function. Then the composite functions
$h\circ{g}:C\longrightarrow{X}\longrightarrow{Y}$ are continuous,
for all $$g:C\longrightarrow{X}$$ with $C$ compact Hausdorff, if
and only if $$h:kX\longrightarrow{Y}$$ is
continuous.\label{Th: Theorem 1}
\end{theorem}

\begin{proof}
If $$h:kX\longrightarrow{Y},$$ and $$g:C\longrightarrow{X}$$ are
continuous, then
$$g:C\longrightarrow{kX}$$ is continuous, and $g\circ{h}$ is continuous.
Conversely, let $h\circ{g}$ be continuous, for all maps
$$g:C\longrightarrow{X},$$ where  $C$ is compact Hausdorff, we
wish to prove that $h:kX\longrightarrow{Y}$ is continuous.  Let
$U$  be open  in $Y$. Then
$$(h\circ{g})^{-1}(U)=g^{-1}\circ{h^{-1}(U)}$$  is open
implies that $h^{-1}(U)$ is open in $kX$ , since $kX$ has the
final topology with respect to all maps
$$g:C\longrightarrow{kX}.$$
\end{proof}

\begin{proposition}\label{Pro1} 
The identity function $$1:kX\longrightarrow{X}$$ is continuous,
for all spaces $X$.
\end{proposition}

\begin{proof}
For any map $$f:C\longrightarrow{kX},$$ $1\circ{f}$ is continuous,
so $1:kX\longrightarrow{X}$ is continuous by the Universal
Property.
\end{proof}

\begin{proposition}\label{Pro 2} 
If $$f:X\longrightarrow{Y}$$ is continuous, and $X$ and $Y$ are
spaces, then $$kf:kX\longrightarrow{kY}$$ where $kf(x)=f(x)$ is
continuous.
\end{proposition}

\begin{proof}
Let $C$ be compact Hausdorff space, and
$$g:C\longrightarrow{X}$$ be continuous. Then
$$f\circ{g}:C\longrightarrow{Y}$$ is continuous. Hence
$$kf\circ{g}:C\longrightarrow{kX}\longrightarrow{kY}$$ is
continuous, for any incoming map $$g:C\longrightarrow{X}.$$ It
follows that $$f\circ{g}:C\longrightarrow{kY}$$ is continuous, and
that $kf$ is continuous by the Universal Property.
\end{proof}

\begin{proposition}\label{Pro 3} 
If $X$ is $k-space$, and $Y$ is any space, then
$f:X\longrightarrow{Y}$ is continuous if and only if
$$f^{'}:X\longrightarrow{kY}$$ is continuous, where
$f^{'}(x)=f(x)$ for all $x\in{X}$.
\end{proposition}

\begin{proof}
Let $f$ be continuous. Then $$f:kX=X\longrightarrow{kY}$$ is continuous by the previous proposition.
\newline
Conversely, let $$f^{'}:X\longrightarrow{kY}$$ be continuous. Then
$$f=1\circ{f^{'}}:X\longrightarrow{Y}$$ is continuous, where
$$1:kY\longrightarrow{Y}$$ is the continuous identity function
(see Proposition 2).
\end{proof}

\begin{proposition}\label{Pro 4} 
If $C$ is a compact Hausdorff space, then a function
$g:C\longrightarrow{X}$ is continuous if and only if
$$g:C\longrightarrow{kX}$$  is continuous.
\end{proposition}

\begin{proof}
The \emph{only if} part follows from the definition of $kX$, as
was explained on page 1.
\newline
Conversely, let $$g:C\longrightarrow{kX}$$ be continuous. Then
$$1:kX\longrightarrow{X}$$ is continuous, and so
$$g:C\longrightarrow{X}$$ is continuous.
\end{proof}

\begin{proposition}\label{Pro 5} 
If $C$ is compact Hausdorff space, then $C$ is $k-space$.
\end{proposition}

\begin{proof}
From Proposition 1, $$1:kC\longrightarrow{C}$$ is continuous, and
the identity function
$$g:C\longrightarrow{kC}$$ is continuous by Proposition 4. Then
$kC=C$, and $C$ is a $k-space$.
\end{proof}

\begin{proposition}\label{Pro 6} 
If $X$ is any space, then $kX=k(kX)$.
\end{proposition}

\begin{proof}
The proof lies in the observation that $kX$, and $k(kX)$ have the
final topology relative to all maps  $$g:C\longrightarrow{X}$$ and
all maps $$g:C\longrightarrow{kX}$$ respectively, and by
Proposition 4 these are the same maps in each case.
\end{proof}

\begin{corollary}\label{Cor 1} 
For any space  $X$, $kX$ is a $k-space$.
\end{corollary}

\begin{proof}
From the previous proposition.
\end{proof}
\subsection{CW-Complexes}
\label{sec:CW-complexes}

\begin{proposition}\label{Pro 7} 
If $Y$ has the final topology with respect to a family of
functions
$$\lbrace{f_{j}:X_{j}\longrightarrow{Y}}\rbrace_{j\varepsilon{J}},$$
where all $X_{j}$ are $k-spaces$, then $Y$ is a  $k-space$.
\end{proposition}

\begin{proof}
Let $$g:C\longrightarrow{X_{j}}$$ be continuous, for all
$j\in{J}$, with $C$ compact Hausdorff. Then
$$f_{j}\circ{g}:C\longrightarrow{Y}$$ is continuous, and  if
$U$ is open in $kY$, then $$(f_{j}\circ{g})^{-1}(U)$$ is open in
$C$, by the definition of final topology. Thus
$$g^{-1}(f^{-1}_{j}(U))$$ is open in $C$. Hence
$f^{-1}_{j}(U)$ is open in $kX_{j}=X_{j}$ for all
$j\varepsilon{J}$ by the definition of final topology. Then $U$ is
open in $Y$, again by the definition of final topology, and so
$Y=kY$ as we required.
\end{proof}

\begin{corollary}\label{Cor 2} 
If $$f:X\longrightarrow{Y}$$ is an identification, and $X$ is a $k-space$, then $Y$ is a $k-space$.
\end{corollary}

\begin{corollary}\label{Cor 3} 
If $\lbrace{X_{j}}\rbrace_{j\in{J}}$ is a family of $k-spaces$, then the disjoint topological sum  $$\amalg_{j\varepsilon{J}}X_{j}$$ is 
$k-space$.
\end{corollary}

\begin{corollary}\label{Cor 4} 
Every $CW-complex$ is $k-space$.
\end{corollary}

\begin{proof}
If $\lbrace{D_{j}}\rbrace_{j\varepsilon{J}}$ are the  cells of a $CW-complex$  $K$, and the inclusion $$D_{j}\hookrightarrow{X}$$ is 
denoted by $i_{j}$, then $K$  has the final topology relative to the family 
$$\lbrace{i_{j}:D_{j}}\longrightarrow{K}\rbrace_{j\varepsilon{J}}.$$ We wish to prove that  $K=kK$. 
We know that the identity  $$kK\longrightarrow{K}$$ is continuous, so we just have to prove the  continuity of the identity 
$1:K\longrightarrow{kK}$. Then  $1\circ{i}$ is a map from a compact Hausdorff  space into $K$ and so, by Proposition 4, is continuous. 
It follows by Theorem 1 that  $$1:K\longrightarrow{kK}$$ is continuous, and so $K=kK$,  and  $K$ is $k-space$.
\end{proof}
\subsection{Initial Topologies on $X$}

\begin{remark}\label{remark 1} 
Let $X$ carry the \emph{initial topology} (see \cite{S. Willard}), relative to the family of functions
$$\lbrace{g_{j}:X\longrightarrow{X_{j}}\rbrace_{j\varepsilon{J}}}.$$
If the spaces $X_{j}$ are $k-spaces$, it does not necessarily
follow that $X$ is a $k-space$. The  product space $Y\times{Z}$
carries the initial topology  relative to the projections
$$p_{1}:Y\times{Z}\longrightarrow{Y},$$ and $$p_{2}:Y\times{Z}\longrightarrow{Z},$$ yet
there is well  known examples in \cite{R. Brown3} and \cite{Clifford H. Dowker}, where $Y$ and $Z$
are $CW-complexes$, yet $Y\times{Z}$ is not a $k-space$. However,
the following result tells us that the $k-ification$ of the usual
sense initial topology is the appropriate model for a $k-space$
initial topology on $X.$
\end{remark}

\begin{theorem}{\textbf{The Universal Property for $k-spaces$ Initial Topologies on $X$}}\label{The 2 Universal Pro initial  topo} 
Let  $\lbrace{X_{j}}\rbrace_{j\varepsilon{J}}$ be a family of $k-spaces$, and $X$ carry the initial topology in the usual sense  
relative to a collection of functions $$\lbrace{g_{j}:X\longrightarrow{X_{j}\rbrace_{j\varepsilon{J}}}}.$$  
Then $kX$ is the initial topology of $X$ in the $k-sense$, as can be seen from the following Universal Property.
\begin{itemize}
\item[(a)] The functions $$g_{j}:kX\longrightarrow{X_{j}}$$ are continuous, and
\item[(b)] If $W$ is a $k-space$ and $$h:W\longrightarrow{X}$$ is a function, 
then $$h:W\longrightarrow{kX}$$ is continuous if and only if the composites $$g_{j}\circ{h}:W\longrightarrow{X_{j}}$$ are continuous, 
for all $j\varepsilon{J}$.
\end{itemize}
\end{theorem}

\begin{proof}
(a) follows from Proposition \ref{Pro 2}, (b) from Proposition \ref{Pro 2}, and the
Universal Property of initial topologies in the usual sense.
\end{proof}

\begin{remark}\label{remark 2} 
If $X$ and $Y$ are sets, then a function
$$\alpha:W\longrightarrow{X\times{Y}}$$ is of the form
$<\alpha_{1},\alpha_{2}>$, where
$$\alpha_{1}:W\longrightarrow{X},$$ and
$$\alpha_{2}:W\longrightarrow{Y}.$$ Thus
$\alpha(w)=<\alpha_{1},\alpha_{2}>(w)=(\alpha_{1}(w),\alpha_{2}(w))$,
for all $w\in{W}$.
\newline
If $W$, $X$ and $Z$ are spaces, then the Universal Property of
products spaces asserts that $\alpha$ is continuous if and only if
$\alpha_{1}$, and $\alpha_{2}$ are continuous.
\newline
We define $X\times_{k}{Y}=k(X\times{Y})$. For $k-spaces$ $X$ and
$Y$, it follows from Theorem \ref{The 2 Universal Pro initial  topo} that $X\times_{k}{Y}$ is the
product of $X$ and $Y$ in the $k-sense$.
\end{remark}

\begin{remark}\label{remark 3} 
Given  maps $$p:X\longrightarrow{B},$$ and
$$q:Y\longrightarrow{B},$$ then we will define the
\emph{pullback space or fibred product space of $X$ and $Y$}, to
be the subspace of $X\times{Y}$ with underlying  set
$$X\sqcap{Y}=\lbrace(x,y)\vert{p(x)=q(y)}\rbrace.$$ In this
situation $$p^{\ast}q:X\sqcap{Y}\longrightarrow{X},$$ and
$$q^{\ast}p:X\sqcap{Y}\longrightarrow{Y}$$ will be denote the
corresponding induced projections. Let $W$ be a space. Then it is
standard that $X\sqcap{Y}$ carries the initial topology relative
to the maps $p^{\ast}q$, and $q^{\ast}p$. The typical map
$$W\longrightarrow{X\sqcap{Y}}$$ will be denoted by
$\langle{h,k}\rangle$, where $h\in{M(W,X)}$ and $k\in{M(W,Y)}$
with $ph=qk$, thus $\langle{h,k}\rangle(w)=(h(w),k(w))$ where
$w\in{W}$.

The $k-ification$ of $X\sqcap{Y}$ will be denoted by
$X\sqcap_{k}{Y}$. It follows from Theorem \ref{The 2 Universal Pro initial  topo}  that $X\sqcap_{k}{Y}$
carries the $k-sense$ initial topologies relative to
$k(p^{\ast}q)$, and $k(q^{\ast}p).$
\end{remark}

\subsubsection{\textbf{Exponential Rules for $k-spaces$}}
\label{sec:exp rules}

If $X$ and $Y$ are spaces, $M(X,Y)$ will denote the set of all
maps from $X$ to $Y$. In cases where it is a topological space, it
should be assumed to have the compact-open topology.
\begin{lemma}\label{Lem 1} 
If $X$ is a $k-space$, and $C$ is compact Hausdorff, then
$X\times{C}$ is a $k-space$.
\end{lemma}

\begin{proof}
We need to prove that the identity function
$$1:X\times{C}\longrightarrow{X\times_{k}{C}}$$ is continuous. The
first step is to show that $X\times{C}$ has the final topology
relative to all maps
$$h\times{1_{C}}:K\times{C}\longrightarrow{X\times{C}}$$ where $K$
is compact Hausdorff, and $h\in{M(K,X)}$.

Let $Z$ be an arbitrary space and
$$f:X\times{C}\longrightarrow{Z}$$ be a function. We will
assume that
$$f\circ({h\times{1_{C}}}):K\times{C}\longrightarrow{Z}$$ is
continuous for every compact Hausdorff spaces $K$, and all $h$ in
$M(K,X)$. It follows by the proper condition for the category of
all topological spaces (see \cite [ Ch. V. Lem. 3.1]{Sze-Tsen Hu}). Then there is
an associated map $$u:K\longrightarrow{M(C,Z)}$$ determined by the
rule
$$u(y)(c)=f\circ{(h\times1_{C})}(y,c)$$
$$=f(h(y),c)$$
$$=(gh(y))(c)$$
where $y\in{K}$ and the function  $$g:X\longrightarrow{M(C,Z)}$$
corresponds to $f$ by the rule $g(x)(c)=f(x,c)$, $x\in{K}$ and
$c\in{C}$. Hence $g\circ{h}=u$ is continuous for all choices of
$K$ and $h$. The Universal Property, associated with the $k-space$
topology on $X$, implies that
$$g:X\longrightarrow{M(C,Z)}$$ is continuous.

The admissible condition for the category of all spaces 
[ Ch.V. Cor.3.5 \cite{Sze-Tsen Hu}] now ensures that $f$ is continuous. Hence the maps
$$h\times{1_{C}}:K\times{C}\longrightarrow{X\times{C}}$$ satisfy
the Universal Property associated with the required final topology
on $X\times{C}$, so $X\times{C}$ has that topology.

We will again assume that $K$ is a compact Hausdorff space and
that $h:K\longrightarrow{X}$ is a map. Then
$$h\times{1_{C}}:K\times{C}\longrightarrow{X\times{C}}$$ and
$$h\times{1_{C}}:k(K\times{C})\longrightarrow{k(X\times{C})}=X\times{C}$$
are continuous. Now $K\times{C}$  is compact Hausdorff, so it is a
$k-space$, i.e. $k(X\times{C})=X\times{C}$; hence
$$h\times{1_{C}}:k(K\times{C})\longrightarrow{X\times_{k}{C}}$$
is continuous.

Now this last map is the composite
$$\xymatrix{
  K\times{C} \ar[dr]_{1\circ(h\times{1_{C}})} \ar[r]^{h\times{1_{C}}}
                & X\times{C} \ar[d]^{1}  \\
                & X\times_{k}{C}             }$$
so it follows by the Universal property established earlier in
this proof, that $$1:X\times{C}\longrightarrow{X\times_{k}{C}}$$
is continuous.

Hence $X\times{C}=X\times_{k}{C}$, and so is $k-space$.
\end{proof}

\begin{theorem}\label{The 3} 
Let $X,$ $Y$ and $Z$ be  $k-spaces$.  Then
$$f:X\times_{k}{Y}\longrightarrow{Z}$$ is a continuous function if
and only if $$g:X\longrightarrow{kM(Y,Z)}$$ is continuous, where
$f(x,y)=g(x)(y)$  for all $x\in{X}$, $y\in{Y}$, and  $M(Y,Z)$ is
the space of continuous functions from $Y$ to $Z$ with the compact
open topology.
\end{theorem}

\begin{proof}
The proof follows immediately from the next three results.
\end{proof}
\subsubsection{\textbf{Proper Condition for $k-spaces$}}
\label{sec:proper condition k-spaces}

\begin{proposition}{\textbf{The Proper Condition}}\label{Pro 8 Proper Condition} 
Let $X$, $Y$ and $Z$ be $k-spaces$, and
$f:X\times_{k}Y\longrightarrow{Z}$  be continuous. Then the rule
$g(x)(y)=f(x,y)$, where $x\in{X}$ and $y\in{Y}$, determines a well
defined and continuous function $$g:X\longrightarrow{kM(Y,Z)}.$$
\end{proposition}

\begin{proof}
Fixing $x\in{X}$, let $$g(x):Y\longrightarrow{Z}$$ be defined by
$g(x)(y)=f(x,y)$ where $y\in{Y}$. Then $g(x)$ is clearly a well
defined function. Now we need to prove that $g(x)$ is continuous.
If  $c_{x}:X\longrightarrow{Y}$ is the constant map at value $x$,
then
$$<c_{x},1_{Y}>:Y\longrightarrow{X\times_{k}{Y}}$$
defined by $<c_{x},1_{Y}>(y)=(x,y)$ is a continuous function (see
Remark \ref{remark 2}). It follows that $g(x)=f\circ{<c_{x},1_{Y}>}$ is
continuous.

Let $C$ be compact Hausdorff space, and $\alpha\in{M(C,X)}$. We
wish to prove that $g\circ{\alpha}$ is continuous for all choices
of $\alpha$. For then, by the Universal Property associated with
the k-topology on $X$, $g$ is continuous.

Thus
$$\alpha\times1_{Y}:C\times{Y}\longrightarrow{X\times{Y}}$$ is
continuous, and  $k(\alpha\times1_{Y})$  is continuous. So that
$$f\circ{k(\alpha\times{1_{Y}}):C\times{Y}\longrightarrow{Z}}$$ is
continuous by previous Lemma.  It follows  by the proper condition
in the ordinary sense; see \cite[Lem. 3.1, Pg. 158]{Sze-Tsen Hu}, that
$$h:C\longrightarrow{M(Y,Z)}$$ is continuous, where
$$h(c)(y)=f(\alpha(c),y)=g(\alpha(c))(y),$$ where $c\in{C}$ and $y\in{Y}$. Hence
$$h(c)=g(\alpha(c))=(g\circ{\alpha})(c).$$
So  $g\circ{\alpha}=h$ is continuous for all $\alpha\in{M(C,X)}$,
and the result follows, as explained earlier.
\end{proof}

\begin{proposition}\label{Pro 9} 
If Y and Z are $k-spaces$, then
$$e:kM(Y,Z)\times_{k}Y\longrightarrow{Z}$$ is continuous.
\end{proposition}

\begin{proof}
Given that $C$ is compact Hausdorff, and
$$\alpha:C\longrightarrow{kM(Y,Z)\times_{k}{Y}}$$ is continuous. We
want to prove that $e\circ{\alpha}$ is continuous, where
$$\alpha(c)=(\alpha_{1}(c),\alpha_{2}(c)),$$
$\alpha_{1}:C\longrightarrow{kM(Y,Z)},$ and
$\alpha_{2}:C\longrightarrow{Y}$ are continuous. Further, it
follows by Proposition 4 that
$\alpha_{1}:C\longrightarrow{M(Y,Z)}$ is also continuous, and
$$\alpha^{\ast}_{2}:M(Y,Z)\longrightarrow{M(C,Z)},$$
$\alpha^{\ast}(h)=h\circ{\alpha_{2}}$ is continuous, where
$h\in{M(Y,Z)}$. Now
$$e_{C}:M(C,Z)\times{C}\longrightarrow{Z}$$ is continuous since
$C$ is compact Hausdorff see [Ch.V, Lem. 3.9 \cite{Sze-Tsen Hu}]. Then
$e\circ{\alpha}$ is continuous because
$e\circ{\alpha}=e<\alpha_{1},\alpha_{2}>=e_{C}<\alpha^{\ast}_{2}\circ{\alpha_{1}},1_{C}>$.
Hence $e$ is continuous.
\end{proof}
\subsubsection{\textbf{Admisible Condition}}
\label{sec:Admisible condition}

\begin{proposition}{\textbf{The Admissible Condition}}\label{Pro 10 Admisible Condition} 
If $X$, $Y$ and $Z$ are $k-spaces$, and $g:X\longrightarrow{kM(Y,Z)}$ is continuous, then 
$f:X\times_{k}Y\longrightarrow{Z}$ is continuous where $f(x,y)=g(x)(y)$.
\end{proposition}

\begin{proof}
The proof follows because $f$ is the composite

$$X\times_{k}Y\stackrel{g\times_{k}{1_{Y}}}\longrightarrow{kM(Y,Z)\times_{k}Y}\stackrel{e}\longrightarrow{Z},$$
in which $$e(g\times_{k}{1_{Y}})(x,y)=e(g(x),1_{Y}(y))$$
$$=e(g(x),y)$$ $$=g(x)(y)$$ $$=f(x,y).$$ Hence $f$ is
continuous.
\end{proof}

\begin{theorem} [Ch.V, Th.3.9 \cite{Sze-Tsen Hu}]
Let $X$ and $Y$ be Hausdorff spaces and $Z$ an arbitrary space. If
either of the following conditions are satisfied:
\begin{itemize}
\item [(a)] $Y$ is locally compact or \item[(b)] $X$ and $Y$
satisfy the first axiom of countability,
\end{itemize}
then $$f:X\times{Y}\longrightarrow{Z}$$ is a continuous function
if and only if $$g:X\longrightarrow{M(Y,Z)}$$ is continuous, where
$f(x,y)=g(x)(y)$  for all $x\in{X}$, $y\in{Y}$, and $M(Y,Z)$ is
the space of continuous functions from $Y$ to $Z$ with the compact
open topology.
\end{theorem}

We notice that the inconvenient assumptions built into the
admissible condition of Theorem 4 avoided in the corresponding
result here i.e. Theorem 3.
\begin{theorem}
If $$q:Y\longrightarrow{B}$$ is a Hurewicz fibration in the sense
of usual category of spaces, then $$kq:kY\longrightarrow{kB}$$ is
a Hurewicz
 fibration in the $k-sense$.
\end{theorem}
\begin{proof}
Let $A$ be a $k-space$, and
$$f:A\times{\lbrace{0}\rbrace}\longrightarrow{kY},$$ and
$$F:A\times{I}\longrightarrow{kB}$$ be maps such that
$F(a,0)=kq(f(a,0))$ for all $a\in{A}$. We wish to prove that there
is a map $$F^{\sim}:A\times{I}\longrightarrow{kY}$$ such that
$F^{\sim}(a,0)=f(a,0)$ for $a\in{A}$, and $kp\circ{F^{\sim}}=F.$
\newline
Taking $1_{Y}$, and $1_{B}$ to be the identity functions
$kY\longrightarrow{Y}$, and $kB\longrightarrow{B}$, respectively,
then
$$1_{Y}\circ{f}:A\times{\lbrace{0}\rbrace}\longrightarrow{Y},$$
and $$1_{B}\circ{F}:A\times{I}\longrightarrow{B}$$ are maps such
that $1_{B}\circ{F(a,0)}=q\circ{(1_{Y}\circ{f})(a,0)}$, for all
$a\in{A}$. Then it follows from the covering homotopy property for
$p$ that we can find a map
$$H:A\times{I}\longrightarrow{Y}$$ such that
$1_{B}\circ{F}=p\circ{H}$ and $H(a,0)=1_{Y}f(a,0)$, for all
$a\in{A}$. We define
$$F^{\sim}:A\times{I}\longrightarrow{kY}$$ as having the same
underlying function as $H$.  Now $A\times{I}$ is a $k-space$
by Lemma \ref{Lem 1}, so $H$ is continuous by Proposition \ref{Pro 3}.  
Hence the result follows.
\end{proof}

\section{Mapping Spaces and Fibrewise Homotopy Theory}
\label{sec:fiberwise}

\begin{definition}
A topological space  $B$ is said to be weak Hausdorff if $$\bigtriangleup_{B}=\lbrace(b,b)
\mid{b}\in{B}\rbrace \subset{B\times{B}},$$ is closed in $B\times_{k}B$.
\end{definition}

\begin{definition}\label{Def: Partial Maps} 
If $Z$ is a space, we will define $Z^{\sim}$ as the set
$Z\cup\lbrace{\omega}\rbrace$ where $w\not\in{Z}$. We give $Z^{\sim}$
the topology whose closed sets are $Z^{\sim}$ itself, and the
closed sets of $Z$.  Let $C$  be a closed subspace of  $Y$, and
$$f:C\longrightarrow{Z}$$ be a map, so $f$ is a partial map from
$Y$ to $Z$. Then  there is an associated map
$$f^{\sim}:Y\longrightarrow{Z^{\sim}}$$ defined by the rule
$$
f^{\sim}(y)=
\begin{cases}
f(y) & \text{if $ y\in C$}\\
\omega & \hbox{ otherwise\;.}
\end{cases}
$$

\end{definition}

\begin{remark}\label{remark 4} 
Let $B$ be a $T_{1}-space$, and $q:Y\longrightarrow{B}$ be a map.
We define the set
$$Y!Z=\bigcup_{b\in{B}}M(Y\mid{b},Z)$$ where  $q^{-1}(b)=Y\mid{b}$, and the function
$$q!Z:Y!Z\longrightarrow{B}$$ is the projection that sends all maps $$Y\mid{b}\longrightarrow{Z}$$
to $b$, for all $b\in{B}$.  We know that  $B$ is a $T_{1}-space$,
so each fibre $Y\mid{b}$ is closed in $Y$. It follows that if
$f\in{M(Y\mid{b},Z)}$, then $i(f)=f^{\sim}$ defines a function
$$i:Y!Z\longrightarrow{M(Y,Z^{\sim})}.$$ We define
the modified compact-open topology on $Y!Z$ as being the initial
topology relative to $i$, and $q!Z$.  Thus we define the free
range mapping space $Y!Z$  as  having a subbase consisting of all
sets of the form $(q!Z)^{-1}(U)$, where $U$ is open in $B$, and
all sets of the form
$$W(A,V)=\lbrace{f}\in{Y!Z}\mid{f(A\cap{dom(f))\subset{V}}}\rbrace,$$
where $A$ ranges over the compact subsets of $Y$, and $V$ ranges
over the open subsets of $Z$.

We now introduce a $k-version$ of the free range mapping space
$Y!Z$, i.e.  $k(Y!Z)$.  Thus this space carries the initial
topology relative to $k(q!Z)$, and $k(i)$ in the sense of
$k-spaces$, i.e. the $k-ification$ of the previously defined
topology on $Y!Z$.
\end{remark}

\begin{remark}\label{remark 5} 

\textbf{For the rest of this chapter  all spaces used should be
assumed to be $k-spaces$, and all constructions and definitions
should be understood in that sense. Thus $X\times{Y}$,
$X\sqcap{Y}$, $M(X,Y)$, $X!Y$ and $p!Y$ will refer to concepts
that were previously denoted by $X\times_{k}{Y}$,
$X\sqcap_{k}{Y}$, $kM(X,Y)$, $k(X!Y)$ and $k(p!Y)$, respectively.
The term Hurewicz fibration refers to a map between $k-spaces$
that has the covering homotopy property relative to incoming maps
from $k-spaces$.}
\end{remark}
\subsection{Fibred Exponential Law for $k-spaces$}
\label{sec:Fiwer-expon}

\begin{theorem} {\textbf{ Fibred Exponential Law for k-spaces}}\label{Th:Fibred exp law k-spaces} 
Let $X$, $Y$, $Z$ and $B$  be  $k-spaces$,  with $B$ weak
Hausdorff, and $p:X\longrightarrow{B}$ $q:Y\longrightarrow{B}$,
and $r:Z\longrightarrow{B}$  be maps.  Then there is a bijective
correspondence between
\begin{itemize}
\item[(a)] maps $$f^{>}:X\sqcap{Y}\longrightarrow{Z},$$ and
\item[(b)] fibrewise maps $$f^{<}:X\longrightarrow{Y!Z}$$
determined by the  rule $f^{>}(x,y)=f^{<}(x)(y)$ where
$p(x)=q(y)$.
\end{itemize}
\end{theorem}

\begin{proof}
There is a map
$$p\times{q}:X\times{Y}\longrightarrow{B}\times{B},$$  and the weak
Hausdorff condition ensures that $\bigtriangleup_{B}$ is closed in
$B\times{B}$. Hence the underlying set of $X\sqcap{Y}$, i.e.
$$(p\times{q})^{-1}(\bigtriangleup_{B}),$$ is a closed subspace of
$X\times{Y}$, so it follows that our theory of partial maps from
$Y$ to $Z$, with closed domains, is relevant to the situation
under consideration.

Let
$$f^{>}:X\sqcap{Y}\longrightarrow{Z}$$ be a map.  Then $f^{>}$
determines a map
$$g^{>}=(f^{>})^{\sim}:X\times{Y}\longrightarrow{Z^{\sim}}$$
by the rule $g^{>}(x,y)=f^{>}(x,y)$, where  $p(x)=q(y)$, and
$g^{>}(x,y)=w$ otherwise.  We know by the proper condition
(Proposition 8) that there is an associated map
$$g^{<}:X\longrightarrow{M(Y,Z^{\sim})}$$  defined by
$g^{<}(x)(y)=g^{>}(x,y)$, where $x\in{X}$, and $y\in{Y}$. So
$$g^{<}(x)(y)=w$$ if and only if $$p(x)\neq{q(y)}.$$  We now
define $$f^{<}:X\longrightarrow{Y!Z}$$ by
$f^{<}(x)(y)=g^{<}(x)(y)$  only in the case where $p(x)=q(y)$.
Then
$$f^{<}(x)(y)=g^{<}(x)(y)=g^{>}(x,y)=f^{>}(x,y).$$  However,
$f^{<}(x)(y)$ is undefined when $p(x)\neq{q(y)}$.  If $p(x)=b$,
then $f^{<}(x)(y)$ is defined for all $y\in{Y\mid{b}}$.  i.e.
$f^{<}(x)\in{Y!Z}$, and $(q!Z)(f^{<}(x))=b$.  So
$(q!Z)\circ{f^{<}}=p$, and $(q!Z)\circ{f^{<}}$ is continuous.
Also, recalling  our definition of the topology on $Y!Z$,
$i\circ{f^{<}}=g^{<}$  is continuous.  It  follows by the
Universal Property of the $k-sense$ initial topology on $Y!Z$,
and Proposition 3, that $f^{<}$ is continuous. The argument is
reversible, and so the proof is complete.
\end{proof}
\begin{remark}\label{remark 6} 
If $X$ and $Y$ are spaces, then $[X,Y]$ will denote the set of
homotopy classes of free maps from $X$ to $Y$. If $X$ and $Y$ are
based spaces, then $M^{o}(X,Y)$ will denote the set of based maps
from $X$ to $Y$, with the of course $k-ified$ compact-open
topology, and $[X,Y]^{o}$ will be denote the set of based homotopy
classes. If $Y$ and $B$ are based spaces, and
$q:Y\longrightarrow{B}$ is a map, the set of based sections to q,
i.e.
$$Sec^{o}(q)=\lbrace{f}\in{M^{o}(B,Y)}\mid
q\circ{f}=1_{B}\rbrace$$ will be equipped with the of course
$k-ified$ compact-open topology.
\newline In addition, if $B$ and $Z$ have basepoints $b_{o}\in{B}$ and $z_{o}\in{Z}$, then the constant
map $$c_{z_{o}}:Y\mid{b_{o}}\longrightarrow{Z}$$ is defined by
$c_{z_{o}}(y)=z_{o}$. We take $c_{z_{o}}$  as basepoint for $Y!Z$.
The space $M(X,A;Y,B)$ will denote the set of maps from $X$ to $Y$
for which $f(A)\subseteq{B}$, again with the $k-ified$ compact-open
topology, and $[X,A;Y,B]$ for the corresponding set of homotopy
classes.
\end{remark}
\subsection{Vertical Homotopies and Sections}
\label{sec:vertical homotopy and sections}

\begin{definition} {\textbf{Vertical Homotopy}}
Let $q:Y\longrightarrow{B}$ be a map, and $\ell_{o}$, and
$\ell_{1}$ be sections to $q$. A homotopy
$F:B\times{I}\longrightarrow{Y}$ such that
$F_{t}=F(-,t):B\longrightarrow{Y}$ is a section to $q$, for all
$t\in{I}$, will be said to be a vertical homotopy. The sections
$\ell_{o}$, and $\ell_{1}$ will be said to be vertically homotopic
if there is a vertical homotopy from $\ell_{o}$ to $\ell_{1}$. I
will write $\Pi_{o}(Sec^{o}(q))$ for the corresponding set of
homotopy classes.
\end{definition}

\begin{corollary}{\textbf{Section Rule.}}
Let $(Z,z_{o})$, and $(B,b_{o})$ be based spaces, $B$ being weak
Hausdorff, and $q:Y\longrightarrow{B} $ be a map.
\begin{itemize}
\item [(a)] If $\ell:(Y,Y\vert{b_{o}})\longrightarrow{(Z,z_{o})}$
is a map, then the rule
$\ell^{\bullet}{(b)}=\ell\vert({Y}\vert{b})\longrightarrow{Z}$,
where $b\in{B}$, defines a  based section $\ell^{\bullet}$  to
$q!Z$. Thus $\ell^{\bullet}{(b)}(y)=\ell(y)$,  where $q(y)=b$.
Then there is a bijective correspondence
$$\theta:M(Y,Y\vert{b_{o}};Z,z_{o})\longrightarrow{Sec^{o}{(q!Z)}},$$
where $\theta(\ell)=\ell^{\bullet}$,
$\ell\in{M(Y,Y\vert{b_{o}};Z,z_{o}})$. \item [(b)]  If
$$\ell_{o},\ell_{1} \in{M(Y,Y\vert{b_{o}};Z,z_{o})}$$ then
$\ell_{o}\simeq\ell_{1}$ via homotopy
$$F:(Y\times{I} ;(Y\vert{b_{o}})\times{I})\longrightarrow{(Z,z_{o})},$$ if and only if $\ell^{\bullet}_{o}\simeq\ell^{\bullet}_{1}$
via a based vertical homotopy. \item[(c)] The rule
$[\ell]\rightsquigarrow{[\ell^{\bullet}]}$ defines a bijection
$$\lambda:[Y,Y\vert{b_{o}};Z,z_{o}]\longrightarrow{\Pi_{o}(Sec^{o}(q!Z))}.$$
\end{itemize}
\end{corollary}

\begin{proof}
\item[(a)] The domain of $\ell^{\bullet}(b)$ is $Y\vert{b}$ so
$q\circ{\ell}=1_{B}.$ Also
$\ell^{\bullet}(b_{o})=\ell\vert(Y\vert{b_{o}})=c_{z_{o}}$, so
$\ell^{\bullet}$ is base point preserving. If $B\sqcap{Y}$ is
defined as the pullback of $1_{B}$,  and
$$q:Y\longrightarrow{B},$$ then the projection
$$\pi:B\sqcap{Y}\longrightarrow{Y}$$ is a homeomorphism. So we have
a bijective correspondence between maps
$$\ell:Y\longrightarrow{Z}$$  maps
$$\ell\circ{\pi}:B\sqcap{Y}\longrightarrow{Z}$$ and, by the Fibred
Exponential Law,   $$\ell^{\bullet}:B\longrightarrow{Y!Z}.$$ We
notice that $\ell^{\bullet}(b)(y)=\ell\circ\pi(b,y)=\ell(y),$ were
$q(y)=b.$ \item[(b)] It follows by arguments similar to those the
proof of (a) that $$F:(Y\times{I} ;
(Y\vert{b_{o}})\times{I})\longrightarrow{(Z,z_{o})}$$ is
continuous, if and only if $$G:(B\times{I},
\lbrace{b_{o}}\rbrace\times{I})\longrightarrow{(Y!Z,c_{z_{o}})}$$
is continuous, where $F(y,t)=G(b,t)(y)$, for all $y\in{Y}$,
$t\in{I}$ and $b=q(y)$. Moreover,
$F(Y\vert{b_{o}}\times{I})=z_{o}$  if and only if
$G(b_{o}\times{I})(y)=z_{o}$  for all $y\in{Y\vert{b_{o}}}$ and
$$(B\sqcap{Y})\times{I}\cong(B\times{I})\sqcap{Y},$$ thus
$\ell^{\bullet}_{o}\simeq\ell^{\bullet}_{1}$ as required.
\item[(c)] This follows easily from (a) and (b).
\end{proof}
\begin{remark}\label{remark 7} 
As an example, let $q:Y\longrightarrow{B}$ be a map, $Z$ be a
space where $z_{o}\in{Z}$.  Then there is a function
$$\sigma_{z_{o}}:B\longrightarrow{Y!Z}, $$ where $\sigma_{z_{o}}(b)$
is the constant map from $Y|b\longrightarrow{Z}$ with value
$z_{o}$.

Now $\sigma_{z_{o}}$ corresponds, via Corollary 5, part(a), to the
constant map $Y\longrightarrow{Z}$ value $z_{o}$. Hence
$\sigma_{z_{o}}$ is continuous.  It is easily seen that it is also
a section to $q!Z$.

We now introduce some fibrewise terminology. Fibrewise spaces in
the free sense are simply maps of spaces into $B$.  Let
$p:X\longrightarrow{B}$ and $q:Y\longrightarrow{B}$ be fibrewise
spaces in the free sense. Then a fibrewise map from
$p:X\longrightarrow{B}$ to $q:Y\longrightarrow{B}$ in the free
sense is, of course, a map $f:X\longrightarrow{Y}$ such that
$q\circ{f}=p$.

A fibrewise space in the based sense is a pair $(p,s)$, where
$p:X\longrightarrow{B}$ is a map, and $s:B\longrightarrow{X}$ is a
section to $p$. The reader can observe that if $B$ is a point
$\ast$, then $s:\ast\longrightarrow{X}$ is essentially just the
point $s(\ast)\in{X}$, so
$(p:X\longrightarrow{\ast},s:\ast\longrightarrow{X})$ is
essentially just the based space $(X,s(\ast))$.

If $(p,s)$ and $(q,t)$ are fibrewise based spaces, then
$(p\sqcap{q}, \langle{s,t}\rangle)$ is also a fibrewise based
space.

A fibrewise map in the based sense, from $(p,s)$ to $(q,t)$ is a
map $f:X\longrightarrow{Y}$ such that $q\circ{f}=p$ and
$f\circ{s}=t$. The set of based maps of this sort will be denoted
by $M_{B}(X,Y)$.
\end{remark}
\begin{definition}
If $f,g\in{M_{B}(X,Y)}$ then a fibrewise based homotopy from $f$
to $g$ is a fibrewise map $F:X\times{I}\longrightarrow{Y}$, and
based homotopy such that $F(x,0)=f(x)$ and $F(x,1)=g(x)$, for all
$x\in{X}$.

Thus a fibrewise based homotopy from $f$ to $g$ is just a homotopy
in the ordinary sense from $f$ to $g$, which is a fibrewise based
map at each stage of the deformation. We then write
$f\simeq_{B}{g}$.
\end{definition}

\begin{definition}
The fibrewise tertiary system $(q,s,m)$ consists of a fibrewise
based space $Y$ over $B$, i.e. a  map $q:Y\longrightarrow{B}$, and
a section $s:B\longrightarrow{Y}$ to $q$, and a fibrewise based
map $m:Y\sqcap{Y}\longrightarrow{Y}$, called a fibrewise
multiplication.
\end{definition}
\begin{definition}
The fibrewise multiplication $m$ is fibrewise homotopy commutative
if $m\simeq_{B}{m\circ{\tau}}$, where $\tau$ is the switching
fibrewise homeomorphism
$\tau:Y\sqcap{Y}\longrightarrow{Y\sqcap{Y}}$ defined by
$\tau(y,y^{'})=(y^{'},y)$, for $(y,y^{'})\in{Y\sqcap{Y}}.$
\end{definition}

\begin{definition}
The fibrewise multiplication $m$ is fibrewise homotopy associative
if $$m(m\sqcap{1_{Y}})\simeq_{B}m(1_{Y}\sqcap{m})$$ where
$$Y\sqcap{Y}\sqcap{Y}\stackrel{{1_{B}}\sqcap{m}}\longrightarrow{Y\sqcap{Y}}
\stackrel{m}\longrightarrow{Y},$$  and
$$Y\sqcap{Y}\sqcap{Y}\stackrel{m\sqcap{1_{B}}}\longrightarrow{Y\sqcap{Y}}
\stackrel{m}\longrightarrow{Y}.$$
\end{definition}

\begin{definition}
The fibrewise multiplication $m$ has a fibrewise homotopy
identity, or satisfies the Hopf condition if
$$m(1_{Y}\sqcap{(s\circ{q})})\triangle\,\simeq_{B}\,1_{Y}\,\simeq_{B}\,m((s\circ{q})\sqcap{1_{Y}})\triangle,$$
where $\triangle:Y\longrightarrow{Y\sqcap{Y}}$ denotes the
diagonal map,
$$Y\stackrel{\triangle}\longrightarrow{Y\sqcap{Y}}\stackrel{1_{Y}\sqcap{(s\circ{q}})}\longrightarrow{Y\sqcap{Y}}\stackrel{m}\longrightarrow{Y},$$ and
$$Y\stackrel{\triangle}\longrightarrow{Y\sqcap{Y}}\stackrel{(s\circ{q})\sqcap{1_{Y}}}\longrightarrow{Y\sqcap{Y}}\stackrel{m}\longrightarrow{Y}.$$

\end{definition}

\begin{definition}
The fibrewise based map $\mu:Y\longrightarrow{Y}$ is a fibrewise
homotopy inversion for the fibrewise multiplication $m$ if
$$m(1_{Y}\sqcap{\mu})\triangle\,\simeq_{B}\,s\circ{q}\,\simeq_{B}\,m({\mu}\sqcap{1_{Y}})\triangle,$$
where
$$Y\stackrel{\triangle}\longrightarrow{Y\sqcap{Y}}\stackrel{1_{Y}\sqcap{\mu}}\longrightarrow{Y\sqcap{Y}}\stackrel{m}\longrightarrow{Y},$$ and
$$Y\stackrel{\triangle}\longrightarrow{Y\sqcap{Y}}\stackrel{\mu\sqcap{1_{Y}}}\longrightarrow{Y\sqcap{Y}}\stackrel{m}\longrightarrow{Y}.$$
A homotopy associative fibrewise tertiary system satisfying the
Hopf condition, and for which the fibrewise multiplication admits
an inversion, is called a fibrewise H-group. If a fibrewise
H-group is fibrewise homotopy commutative, then it will be said to
be homotopy Abelian.\newline More details concerning fibrewise
homotopy are given in a \textit{locus classicus} \cite{P.Booth5}, \cite{P. Booth6}
\cite{I.M. James1}, \cite{I.M. James2}, \cite{D. Pupe}.
\end{definition}
\begin{proposition}
Let $Z$ be an $H-group$, $B$ a weak Hausdorff space, and
$q:Y\longrightarrow{B}$  a map. Then there is a fibrewise map
$$n:Y!Z\,\times\,{Y!Z}\longrightarrow{Y!Z},$$
$n(f,g)=m(f\times{g})\triangle_{b}$, where $b\in{B}$,
$f,g\in{M(Y|b,Z)}$, $m$ denotes the operation on $Z$,
$\triangle_{b}$ is the diagonal map for $Y|b$, and
$m(f\times{g})\triangle_{b}$ is the following composite of maps
$$Y|b\stackrel{\triangle_{b}}\longrightarrow{Y|b\times{Y|b}}\stackrel{f\times{g}}\longrightarrow{Z\times{Z}}\stackrel{m}\longrightarrow{Z}.$$
Then, defining $\sigma_{e}$ as in the example of Remark \ref{remark 7},  the
tertiary system
$$(q!Z,\sigma_{(e)},n)$$ is a fibrewise H-group. Further, if $Z$
is homotopy Abelian, then $(q!Z,\sigma_{e},n)$ is fibrewise
homotopy Abelian.
\end{proposition}
\begin{proof}
If $Y$ is a space and $Z$ is an $H-group$, then the operation
$$n:M(Y,Z)\times{M(Y,Z)}\longrightarrow{M(Y,Z)},$$ defined by
$n(f,g)=n\circ{(f\times{g})}\circ{\triangle_{Y}}$, together with
the identity map $c_{e}:Y\longrightarrow{Z}$, makes $M(Y,Z)$ into
an $H-group$ in an obvious fashion. If $Z$ is homotopy Abelian,
then so also is $M(Y,Z)$. The proof of this proposition is a
direct generalization of that argument, using the fibred
exponential law of Theorem 5, rather than the usual exponential
law for spaces.
\end{proof}
\begin{proposition}
If $(q:Y\longrightarrow{B},\, t:B\longrightarrow{Y},\,
m:Y\sqcap{Y}\longrightarrow{Y})$  is a fibrewise homotopy Abelian
$H-group$, then $Sec^{o}(q)$ is a homotopy Abelian $H-group.$ Thus
if $t_{1},t_{2}\in Sec^{o}(p)$, the operation $+_{B}$ on
$Sec^{o}(q)$ is defined by
$$t_{1}+_{B} t_{2}=m\circ\langle{t_{1},t_{2}}\rangle,$$ and the identity point for $Sec^{o}(q)$
is $t$.
\end{proposition}
\begin{proof}
The proof is routine.
\end{proof}
\begin{corollary}
If $(Z,z_{o})$ is an Abelian $H-group$, $(B,b_{o})$ is based weak
Hausdorff space, and $q:Y\longrightarrow{B}$ is a map, then
$Sec^{o}(q!Z)$ is an Abelian $H-group$, and
$\Pi_{o}(Sec^{o}(q!Z))$ is an Abelian group.
\end{corollary}

\begin{theorem}
If $(Z,z_{o})$ is an Abelian $H-group$, $(B,b_{o})$ is based weak
Hausdorff space, and $q:Y\longrightarrow{B}$ is a map, then
\item[(a)] the set $$[Y,Y\vert{b_{o}};Z,z_{o}]$$ carries an
Abelian group structure, which is defined by pointwise addition of
homotopy classes, and \item[(b)] the bijection of corollary 5,
part (c), i.e.
$$\lambda:[Y,Y\vert{b_{o}};Z,z_{o}]^{o}\approx\Pi_{o}(Sec^{o}(q!Z))$$
is an isomorphism.
\end{theorem}
\begin{proof}
(a) is routine. (b) The two group structures are both induced by
the H-group structure on $Z$; it is routine to verify that, as
expected, $\lambda$ is an isomorphism.
\end{proof}

\begin{proposition}
Let $B$, $Y$ and $Z$ be spaces and $B$ be weak Hausdorff.  If
$q:Y\longrightarrow{B}$ is a Hurewicz fibration, then $q!Z$ is
also a Hurewicz fibration.
\end{proposition}

\begin{proof}
This is just the argument that proves Theorem 4.1 of \cite{P. Booth6}, but
reinterpreted in $k-context$.  We assume that
$$F:A\times{I}\longrightarrow{B}$$ is a homotopy and the restriction $F\mid{A\times{0}}$ is denoted by $F_{o}$.
We then have pullback spaces $(A\times{I})\sqcap{Y}$, and
$(A\times{0})\sqcap{Y}$, induced by the homotopy $F$ and the map
$F_{o}$, respectively, and associate projections
$$F^{\ast}q:(A\times{I})\sqcap{Y}\longrightarrow{A\times{I}},$$
$$(F_{o})^{\ast}q:(A\times{0})\sqcap{Y}\longrightarrow{A\times{I}},$$

$$q^{\ast}F:(A\times{I})\sqcap{Y}\longrightarrow{Y},$$
$$q^{\ast}F_{o}:(A\times{0})\sqcap{Y}\longrightarrow{Y},$$
and such that
$$q\circ{(q^{\ast}F)}=F\circ(F^{\ast}q)$$ and
$$q\circ{(q^{\ast}F_{o})}=F_{o}\circ{(F_{o})^{\ast}}q.$$
We recall that $(A\times{0})\sqcap{Y}$ is a retract of
$(A\times{I})\sqcap{Y}$. The proof of this, in the usual
topological context, is given in [B6, Lm. 4.2]; the $k-case$
proof is similar. Let
$$k^{<}:A\times{0}\longrightarrow{Y!Z}$$ be a map  such that
$(q!Z)\circ{k^{<}}=F_{o}$. It follows, by the Fibred Exponential
Law for $k-spaces$, that there is an associated map
$$k^{>}:(A\times{0})\sqcap{Y}\longrightarrow{Z}$$ defined by
$k^{>}(a,0,y)=k^{<}(a,0)(y)$ where
$(a,0,y)\in({A}\times{0})\sqcap{Y}$. Now $(A\times{0})\sqcap{Y}$
is known to be a retract of $(A\times{I})\sqcap{Y}$ (compare with
\cite[Lemma 4.2]{P. Booth6}).

Let
$$R:(A\times{I})\sqcap{Y}\longrightarrow{(A\times{0})\sqcap{Y}}$$
be a retraction.  Then the composite
$$k^{>}\circ{R}:(A\times{I})\sqcap{Y}\longrightarrow{Z}$$
corresponds, via the fibred exponential law to
$$K:A\times{I}\longrightarrow{Y!Z},$$ where
$K(a,t)(y)=(k^{>}\circ{R})(a,t,y)$, and
$(a,t,y)\in(A\times{I})\sqcap{Y}.$ Then $K$ is fibrewise over $B$,
i.e. $(q!Z)K(a,t)=F(a,t),$ so $(q!Z)\circ{K}=F.$  Also if
$(a,0,y)\in(A\times{0})\sqcap{Y}$,
$$K(a,0)(y)=(k^{>}\circ{R})(a,0,y)=k^{>}(a,0,y)=k^{<}(a,0)(y).$$
So $K(a,0)=k^{<}(a,0)$  for $a\in{A}$. i.e. $K$ extends $k^{<}.$
Thus $K$ lifts $F$  and extends $k^{<},$ and $q!Z$ is a Hurewicz
fibration.
\end{proof}

\section{Applications to Homotopy Theory}

We will now compare the main result of \cite{P. Booth6} to the result of our
section. \textbf{In this section we do not assume that spaces
are $k-spaces$, unless we specifically say so.}

\begin{theorem}\textbf{Fibred Exponential Law.} \cite[Th. 3.3]{P. Booth6} Let $B$ be a Hausdorff
space, $Z$ be a  space, and $p:X\longrightarrow{B}$ and
$q:Y\longrightarrow{B}$ be a maps.
\begin{itemize}
\item[(a)] \textbf{Proper Condition:} If
$f^{>}:X\sqcap{Y}\longrightarrow{Z}$ is a map, then the rule
$f^{<}(x)(y)=f^{>}(x,y)$ determines a fibrewise map
$f^{>}:X\longrightarrow{Y!Z}$, where $p(x)=q(y)$.  Thus $f^{<}$ is
a map such that $(q!Z)\circ{f^{<}}=p$. \item[(b)]
\textbf{Admissible Condition:} Let us assume that either
\item[(i)]$(X,Y)$ is an exponential pair of spaces, or
\item[(ii)]$W$ is a space, $p:B\times{W}\longrightarrow{B}$ the
projection, and $Y\times{W}$ a $k-space$. Then, given a fibrewise
map $f^{>}:X\longrightarrow{Y!Z}$, the above rule determines a map
$f^{>}:X\sqcap{Y}\longrightarrow{Z}$.
\end{itemize}
\end{theorem}

We notice that the inconvenient assumptions built into the
admissible condition of Theorem 8 avoided in the corresponding
result here i.e. Theorem 6.

\subsection{Section Rule}
\label{sec:sect-rules}

\begin{corollary}\textbf{Section Rule.} \cite[Cor.3.4]{P. Booth6} Let $B$ be a Hausdorff space, and
$q:Y\longrightarrow{B}$ be a map.
\begin{itemize}
\item[(a)] If $l:Y\longrightarrow{Z}$ is a map, then the rule
$l^{\bullet}(b)=l|(Y|b):Y|b\longrightarrow{Z}$, where $b\in{B}$,
defines a section $l^{\bullet}$ to $q!Z$. Equivalently, we may
define $l^{\bullet}$ by $l^{\bullet}(b)(y)=l(y)$, where $q(y)=b$.
\item[(b)] If $Y$ is $k-space$ and $l^{\bullet}$ is a section to
$q!Z$, then the rule stated in (a) determines a map
$l:Y\longrightarrow{Z}$.
\end{itemize}
\end{corollary}

The corresponding result in this work is Corollary 5.
\begin{theorem}\cite[Th. 8.1(b)]{P. Booth6}
There is a canonical bijection:
$$\theta:H^{m}(Y,Y|b;G)\longrightarrow{\Pi_{o}(Sec^{o}(p!K(G,m)))}$$
where the map $\theta$ is determined by the rule
$\theta[l]=[l^{\bullet}]$, where $l^{\bullet}(b)(y)=l(y)$ and
$q(y)=b$.
\end{theorem}

\subsection{$\Omega-spectrum$ of Eilenberg-MacLane spaces}
\label{sec:Spectrum}

If we follow the $\Omega-spectrum$ of Eilenberg-MacLane spaces to
cohomology, then the associated cohomology groups are defined by
$$H^{m}(Y,Y|b;G)=[Y, Y|b_{o};K(G,m),e],$$ ,for more details
concerning this spectra the reader can see \cite[Def. 8.4.6]{C.R.F. Maunder}.

Corollary 5 of this paper gives the k-version of both Corrolary
7 and Theorem 8, and  our Theorem 6 improves on Theorem 8.1(b) by
showing that the bijection of that result is actually an
isomorphism.

Hence we have established the $k-version$ foundation of the
application to the cohomology of fibrations that is discussed in
Ch. 8 of [B6].
\begin{theorem}\cite[Th. 4.1]{P. Booth6}. 
Let $B$, $Y$ and $Z$ be spaces, where $B$ is Hausdorff and $Y$ is
locally compact Hausdorff.  If $q:Y\longrightarrow{B}$ is a
Hurewicz fibration, then $q!Z$ is also a Hurewicz fibration.
\end{theorem}

The reader  can compare that result with the Proposition 12, and observe that Proposition 12 is free of the
inconvenient assumption that $Y$ is locally compact and Hausdorff.

We now consider the first group of applications given in \cite{P. Booth6},
i.e. Theorems 10, 11 and 12.
\begin{theorem}  \cite[Th. 5.1]{P. Booth6}.
Let $A$ be  a $k-space$ and $B$ be a Hausdorff space. If
$q:Y\longrightarrow{B}$ is an identification and
$f:A\longrightarrow{B}$ is a map, then
$f^{\ast}p:Y\sqcap{A}\longrightarrow{A}$ is an identification.
\end{theorem}
\begin{theorem}  \cite[Th. 6.1]{P. Booth6}.
Let $q:Y\longrightarrow{B}$ be a Hurewicz fibration, where
$B$ is a Hausdorff space and $Y$ is locally compact Hausdorff. If
$A\longrightarrow{B}$ is a closed cofibration, then
$Y|A\longrightarrow{Y}$ is also a closed cofibration.
\end{theorem}

\begin{remark}\label{remark 8} 
Let $q:Y\longrightarrow{B}$ be a map and $t:B\longrightarrow{Y}$
be a section to $q$.  If $f:A\longrightarrow{B}$ is a map, then
$$\sigma:A\longrightarrow{Y\sqcap{A}},$$ defined by
$\sigma{(a)}=(tf(a),a)$, for all $a\in{A}$, is a section to the
projection $f^{\ast}q:Y\sqcap{A}\longrightarrow{A}$.
\end{remark}
\begin{theorem}\cite[Th.7.1]{P. Booth6}.
Let $q:Y\longrightarrow{B}$ be a Hurewicz fibration, with closed
cofibration section $t$, $B$ be Hausdorff and $Y$ locally compact
Hausdorff.  If $f:A\longrightarrow{B}$ is a map, then
$f^{\ast}q:A\sqcap{A}\longrightarrow{A}$ is a Hurewicz fibration
with a closed cofibration section $\sigma$.
\end{theorem}
\begin{remark}\label{remark 9} 
If we modify those proofs by assuming that all spaces are
$k-spaces$ and $B$ is weak Hausdorff, and replacing the about
results by our Theorems 5, we obtain the following analogous
results.
\end{remark}
\begin{theorem}
Let $A$ be  a $k-space$ and $B$ be a weak Hausdorff space. If
$q:Y\longrightarrow{B}$ is an identification and
$f:A\longrightarrow{B}$ is a map, then
$f^{\ast}q:Y\sqcap{A}\longrightarrow{A}$ is an identification.
\end{theorem}

\begin{theorem}
Let $q:Y\longrightarrow{B}$ be a Hurewicz fibration, where $B$ is
a weak Hausdorff space and $Y$ is $k-space$. If
$A\longrightarrow{B}$ is a closed cofibration, then
$Y|A\longrightarrow{Y}$ is also a closed cofibration.
\end{theorem}

\begin{theorem}
Let $q:Y\longrightarrow{B}$ be a Hurewicz fibration, with closed
cofibration section $t$, $B$ be weak Hausdorff and $Y$ $k-space$.
If $f:A\longrightarrow{B}$ is a map, then
$f^{\ast}q:A\sqcap{A}\longrightarrow{A}$ is a Hurewicz fibration
with a closed cofibration section $\sigma$.
\end{theorem}

The reader will notice that the locally compact Hausdorff
assumption of Theorem 12 and 13 have  now been eliminated.
\section{Moore-Postnikov System}

Let $G$ and $H$ be Abelian groups and $m$ and $n$ be integers with
$1<n<m$. Then $$q_{1}:PK(G,m+1)\longrightarrow{K(G,m+1)}$$ will
denote the path fibration over the Eilenberg-MacLane space
$K(G,m+1)$ and $K(H,n+1)$ (see \cite{E.H. Spanier}, Pgs. 75 and 99). Let
$(B,b_{o})$ be a space with a basepoint.  Then a \emph{3-stage
Postnikov tower} $\tau(k_{1},k_{2})=p_{1}\circ{p_{2}}$, over $B$
and with fibres $K(G,m)$ and $K(H,n)$, consists of principal
fibrations
$$p_{1}:E_{1}\longrightarrow{B}$$ and
$$p_{2}:E_{2}\longrightarrow{E_{1}}$$ with fibres  $K(G,m)$ and $K(H,n)$
respectively.  So $p_{1}$ is induced from $q_{1}$ by first
k-invariant  $$k_{1}:B\longrightarrow{K(G,m+1)},$$ i.e.
$p_{1}=k^{\ast}_{1}q_{1}$,  and $p_{2}$ is induced by $q_{2}$ by
the second k-invariant $$k_{2}:B\longrightarrow{K(H,n+1)},$$ i.e.
$p_{2}=k^{\ast}_{2}q_{2}$. The maps $k_{1}$  and $k_{2}$ are based
mappings, where the identities of $K(G,m+1)$ and  $K(H,n+1)$ are
their base points.

The fibre of $p_{1}$ over $b_{o}$ is
$$\{b_{o}\}\times{(PK(G,m)|k_{1}(b_{o}))}=\{b_{o}\}\times{\Omega(K(G,m+1))}=\{b_{o}\}\times{K(G,m)},$$
where $\Omega$ indicate the corresponding loop space.  The fibre
of
$$\tau(k_{1},k_{2})=p_{1}\circ{p_{2}}:E_{2}\longrightarrow{B}$$ is
the subspace of $E_{2}$ obtained by pulling back aver over
$k_{2}|\{b_{o}\}\times{(K(G,m)}$, i.e.
$\{b_{o}\}\times{(K(G,m)}\sqcap{(K(H,n+1)}.$ So the fibre of
$\tau(k_{1},k_{2})$ is $K(G,m)\times{K(H,n)}$ if and only if
$$k_{2}|\{b_{o}\}\times{K(G,m)}:\{b_{o}\}\times{K(G,m)}\longrightarrow{K(H,n+1)}$$
is the constant map
$$c_{e}:\{b_{o}\}\times{K(G,m)}\longrightarrow{K(H,n+1)}$$ with value
the identity $e$ of $K(H,n+1)$.

The problem considered in   \cite{P. Booth6}, and in \cite{P.Booth5},
is the classification of 3-stage Postnikov towers, with fibre
$K(G,m)\times{K(H,n)}$, up to fibrewise homotopy equivalence. The
map $$q:PK(G,m+1)\longrightarrow{K(G,m+1)}$$ and the space
$K(H,n+1)$ enable us to define a free range mapping space
$PK(G,m+1)!K(H,n+1)$.  The classifying space  $M_{\infty}$ used in
\cite{P.Booth5} and \cite{P. Booth6} is just a path component of that space, i.e. the
component that contains the constant map
$$c_{e}:K(G,m)\longrightarrow{K(H,n+1)}.$$

It is shown in \cite[Th.7.5]{P.Booth5}, that $[B,M_{\infty}]^{o}$ classifies
our 3-stage Postnikov tower up to a strong form of fibrewise
homotopy equivalence.

Let $\varepsilon{(K(G,m)\times{K(H,n)})}$ denote the group of
homotopy classes of self-homotopy equivalence of
$K(G,m)\times{K(H,n)}$.  It is shown in \cite{P.Booth7} that an orbit set of
$[B,M_{\infty}]^{o}$, under an action of
$\varepsilon{(K(G,m)\times{K(H,n)})}$, classifies the fibre
homotopy equivalence classes of 3-stage towers as discussed above.

Theorem (M6) of \cite{P.Booth5}. We consider 3-stage towers over $B$ with
fibre   $$K(G,m)\times {K(H,n)},$$  is proved on p.98 of \cite{P.Booth5} using
the general form of the fibrewise exponential law  for $k-spaces$
see  \cite[Th. 5.2]{P.Booth5}. However, it is more natural to prove a free
range exponential law for $k-spaces$ as in our Theorem 6. That
state and proof goes as follows:
\subsection{Classification Theorem}
\label{sec:classification-theorem}

\begin{theorem}
Let $\tau(k_{1},k_{2})=p_{1}\circ{p_{2}}$ be a
$(K(G,m)\times{K(H,n)}$-tower. If $k_{1}\in{M^{o}(B,K(G,m+1))}$
and $g\in{M(E_{1},\{b_{o}\}\times{K(G,m+1)};K(H,n+1),\{e\} )}$,
then there is a bijective correspondence between:
\begin{itemize}
\item [(a)] $K(G,m)\times{K(H,n)}$-towers $\tau(k_{1},k_{2})$, and
\item [(b)] maps $k\in{M^{o}(B,M_{\infty})}$ determined by the
rule $k(b)(l)=k_{2}(b,l)$, where $(b,l)\in
{B\sqcap{PK(G,m+1)}}=E_{1}$, i.e. $k(b)=q_{1}(l)$.
\end{itemize}
\end{theorem}

\begin{proof}
let $\tau(k_{1},k_{2})=p_{1}\circ{p_{2}}$ be a 3-stage Postnikov
tower, over a path connected  and weak Hausdorff space $B$, as
described above. It follows by Theorem \ref{Th:Fibred exp law k-spaces} that
$$k_{2}:B\sqcap{PK(G,m+1)}\longrightarrow{K(H,n+1)}$$ determines
a map $$k:B\longrightarrow{PK(H,m+1)!K(H,n+1)},$$ where
$k(b)(l)=k_{2}(b,l)$ and  $k_{1}(b)=q_{1}(l)$.

Conversely if $k$ is given then $k_{2}$ is determined in this way.

Then $\tau(k_{1},k_{2})$ has fibre $K(G,m)\times{K(H,n)}$ if and
only if  $$k_{2}|\{b_{o}\}\times{K(G,m)}=c_{e},$$ which is true if
and only if  $k(b_{o})(l)=e$, i.e. if
$$k(b_{o})(l)=c_{e}:K(G,m)\longrightarrow{K(H,n+1)}.$$

Now $B$ is path connected, so $k(B)$ is in the path component 
of   $$PK(G,m+1)!K(H,n+1)$$ that contains $c_{e}$, i.e.
$k(B)\subset{M_{\infty}}$. Then the result follows.
\end{proof}

\section{Open Problems and Questions}

\textit{Aknowledgments.} I feel deep sense of gratitude toward to proffesor Peter Booth. 
His professional and personal assistance is invaluable to me.
In particular, he introduced me to this subject matter and helped me transforma the task of writing mathematics into an eye-opening
delight in homotopical invariants.

\end{document}